\documentclass{article}
\usepackage[final]{graphicx}
\usepackage{psfrag}
\usepackage{amsmath,amsfonts,amssymb,amsxtra,subeqnarray,bm}

\newlength{\extramargin}
\newcommand{\Real}{\ensuremath{{\mathbb{R}}}}

\newcommand{\Complex}{\ensuremath{{\mathbb{C}}}}

\newcommand{\C}{\ensuremath{\mathcal C}}

\newcommand{\setL}{\ensuremath{\mathcal L}}

\renewcommand{\L}{\ensuremath{\mathcal L}}

\newcommand{\ex}{\ensuremath{{\mathbf{x}}}}
\newcommand{\zi}{\ensuremath{{\mathbf{z}}}}

\newcommand{\one}{\ensuremath{{\mathbf{1}}}}

\newcommand{\diag}{{\mbox{\rm diag}}}

\newtheorem{theorem}{Theorem}

\newtheorem{corollary}{Corollary}

\newtheorem{lemma}{Lemma}

\newtheorem{fact}{Fact}

\newtheorem{remark}{Remark}

\newtheorem{proposition}{Proposition}
\newenvironment{proof}{\noindent {\bf Proof.}}{\hfill \hspace*{1pt}\hfill$\blacksquare$}

\begin{document}
\title{Synchronization of small oscillations}
\author{S. Emre Tuna\footnote{The author is with Department of
Electrical and Electronics Engineering, Middle East Technical
University, 06800 Ankara, Turkey. Email: {\tt etuna@metu.edu.tr}}}
\maketitle

\begin{abstract}
Synchronization is studied in an array of identical oscillators
undergoing small vibrations. The overall coupling is described by a
pair of matrix-weighted Laplacian matrices; one representing the
dissipative, the other the restorative connectors. A construction is
proposed to combine these two real matrices in a single complex
matrix. It is shown that whether the oscillators synchronize in the
steady state or not depends on the number of eigenvalues of this
complex matrix on the imaginary axis. Certain refinements of this
condition for the special cases, where the restorative coupling is
either weak or absent, are also presented.
\end{abstract}

\section{Introduction}

Consider the dynamics \cite[Ch.~11]{lax96}
\begin{eqnarray}\label{eqn:unit}
M{\ddot x}+Kx=0
\end{eqnarray}
where $x\in\Real^{n}$ and the matrices $M,\,K\in\Real^{n\times n}$
are symmetric positive definite. This linear time-invariant
differential equation, being the generalization of that of harmonic
oscillator, plays an important role in mechanics. It emerges as the
linearization of a Lagrangian system about a stable equilibrium and
satisfactorily represents the behavior of the actual system
undergoing small oscillations \cite[Ch.~5]{arnold89}. Among examples
obeying \eqref{eqn:unit} are the $n$-link pendulum
(Fig.~\ref{fig:pendulum}) and the mass-spring system
(Fig.~\ref{fig:masspring}). It is possible to find relevant systems
outside the domain of mechanics as well. For instance, the LC
circuit shown in Fig.~\ref{fig:LCcircuit} is also described by the
form~\eqref{eqn:unit}; see \cite{tuna16}.

\begin{figure}[h]
\begin{center}
\includegraphics[scale=0.55]{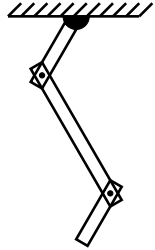}
\caption{3-link pendulum.}\label{fig:pendulum}
\end{center}
\end{figure}

\begin{figure}[h]
\begin{center}
\includegraphics[scale=0.45]{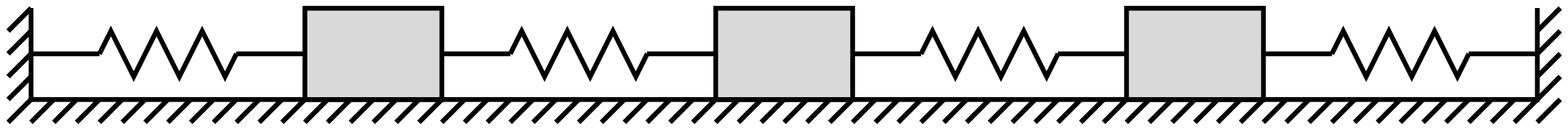}
\caption{Mass-spring system.}\label{fig:masspring}
\end{center}
\end{figure}

\begin{figure}[h]
\begin{center}
\includegraphics[scale=0.45]{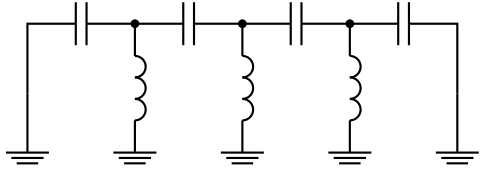}
\caption{LC oscillator.}\label{fig:LCcircuit}
\end{center}
\end{figure}

Suppose now we take a number of identical $n$-link pendulums, each
obeying \eqref{eqn:unit}, and couple them via passive components
such as springs and dampers as shown in Fig.~\ref{fig:pendcoupled}.
Or, we gather a number of identical LC circuits and connect them
through inductors and resistors as shown in
Fig.~\ref{fig:LCcoupled}. What can be said about the collective
behavior of these arrays? In this paper we attempt to answer this
question from the synchronization point of view. That is, we
investigate conditions on the coupling that guarantee asymptotic
synchronization throughout the array, where all the units tend to
oscillate in unison despite the initial differences in their
trajectories.

\begin{figure}[h]
\begin{center}
\includegraphics[scale=0.55]{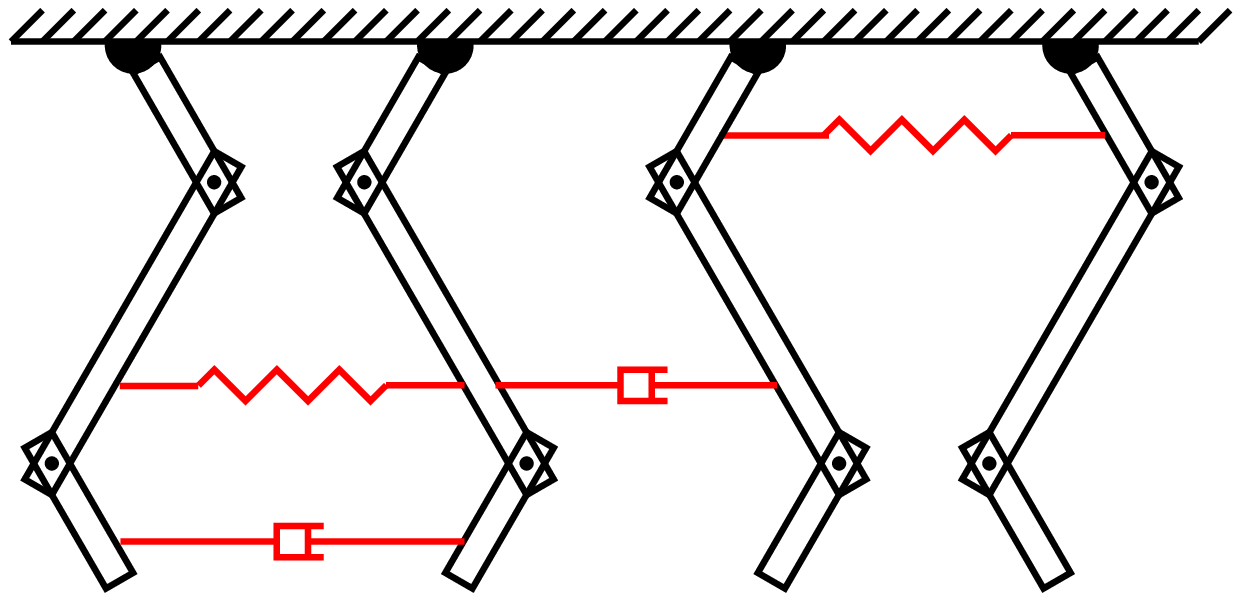}
\caption{Coupled 3-link pendulums.}\label{fig:pendcoupled}
\end{center}
\end{figure}

\begin{figure}[h]
\begin{center}
\includegraphics[scale=0.45]{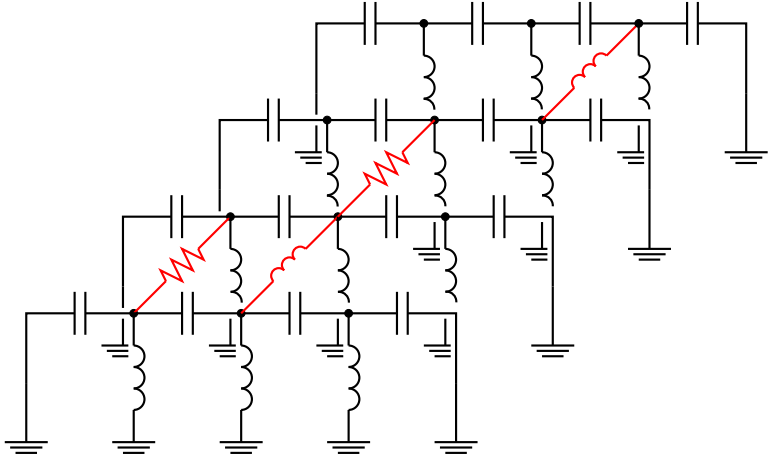}
\caption{Coupled LC oscillators.}\label{fig:LCcoupled}
\end{center}
\end{figure}

In studying synchronization stability the workhorse of the analysis
is the matrix that describes the overall coupling, the ubiquitous
{\em Laplacian}. The classical Laplacian matrix is a very useful
representation of a graph with scalar-weighted edges. This matrix
often appears in various network dynamics and its spectral
properties have proved instrumental in understanding or establishing
synchronization; see, for instance,
\cite{olfati04,li10,dorfler14,eroglu17}. Although a single
scalar-weighted Laplacian turns out to be quite able to represent
the coupling in many different networks (which have been thoroughly
investigated in the duly vast literature) significant exceptions do
exist. One such exception we find appropriate to point out has to do
with the case where the coupling can only be represented by a {\em
matrix-weighted} Laplacian \cite{tuna16,trinh18,tuna18}. Another
instance of deviation manifests itself in the array of harmonic
oscillators linked simultaneously by both dissipative and
restorative connectors \cite{tuna17}, where {\em two} separate
scalar-weighted Laplacians are required to account for the coupling
in its entirety; one for the restorative, the other for the
dissipative links. The particular problem we consider in this paper
happens to fit to neither of these instances and instead contains
them as special cases. Namely, the coupling of the array we study
here cannot be properly described except by a pair of
matrix-weighted Laplacians. To the best of our knowledge, the
problem of synchronization of small oscillations has not yet been
investigated under such direction and degree of generality. It is,
of course, worthwhile to ask whether the suggested generalization is
meaningful. In short, is it (in some sense) natural? We believe that
it is; for two reasons. First, as we mentioned already, the dynamics
we study can be realized by some very basic building blocks from
physics and engineering: pendulum, spring, damper; or, capacitor,
inductor, resistor. Second, some of the methods we develop in our
analysis bear strong resemblance to classical tools from systems
theory and graph theory, such as the Popov-Belevitch-Hautus (PBH)
test for observability and the positivity check of the second
smallest eigenvalue of the Laplacian for connectivity.

Somewhat imprecisely, we now give the statements of the three main
results of this paper. Our setup, the array of $q$ oscillators, is
described by three parameters (matrices): $P,\,L_{\rm d},\,L_{\rm
r}$. (The precise problem statement and notation are given in
Section~\ref{sec:PS}.) The symmetric positive definite matrix
$P\in\Real^{n\times n}$ models the individual oscillator, where $n$
is the number of normal modes or characteristic frequencies. The
matrix-weighted Laplacians $L_{\rm d},\,L_{\rm r}\in\Real^{qn\times
qn}$ represent, respectively, the dissipative coupling (e.g.,
dampers) and the restorative coupling (e.g., springs). Inspired by
how the conductance ($g$) and susceptance ($b$) are brought together
to form the admittance ($y=g+jb$) in circuit theory \cite{desoer69},
we construct from our three matrices the single matrix $[L_{\rm
d}+j([I_{q}\otimes P]+L_{\rm r})]$. In Section~\ref{sec:SS} we
establish the following equivalence between this matrix and
synchrony: {\em The oscillators (asymptotically) synchronize if and
only if $[L_{\rm d}+j([I_{q}\otimes P]+L_{\rm r})]$ has exactly $n$
eigenvalues on the imaginary axis.} To develop a somewhat deeper
understanding of this result we then dissect the matrix-weighted
Laplacians $L_{\rm d},\,L_{\rm r}$ using the eigenvectors
$v_{1},\,v_{2},\,\ldots,\,v_{n}$ of $P$ and obtain the collections
of scalar-weighted Laplacians
$G_{11},\,G_{22},\,\ldots,\,G_{nn}\in\Real^{q\times q}$ and
$B_{11},\,B_{22},\,\ldots,\,B_{nn}\in\Real^{q\times q}$ through
$G_{kk}=[I_{q}\otimes v_{k}^{T}]L_{\rm d}[I_{q}\otimes v_{k}]$ and
$B_{kk}=[I_{q}\otimes v_{k}^{T}]L_{\rm r}[I_{q}\otimes v_{k}]$.
These matrices are employed in Section~\ref{sec:WR} to show: {\em
For weak enough restorative coupling} ($\|L_{\rm r}\|\ll 1$) {\em
the oscillators synchronize if every $[G_{kk}+jB_{kk}]$ has a single
eigenvalue on the imaginary axis.} Finally, in Section~\ref{sec:PD},
we study the pure dissipative coupling scenario. There we find: {\em
In the absence of restorative coupling $(L_{\rm r}=0)$ the
oscillators synchronize if and only if every $G_{kk}$ has a single
eigenvalue at the origin.}

\section{Problem statement and notation}\label{sec:PS}

Consider the array of $q$ coupled oscillators (each of order $2n$)
of the form
\begin{eqnarray}\label{eqn:array}
M{\ddot x}_{i}+Kx_{i}+\sum_{j=1}^{q}D_{ij}({\dot x}_{i}-{\dot
x}_{j})+\sum_{j=1}^{q} R_{ij}(x_{i}-x_{j})=0\,,\qquad
i=1,\,2,\,\ldots,\,q
\end{eqnarray}
where $x_{i}\in\Real^{n}$ and
$M,\,K,\,D_{ij},\,R_{ij}\in\Real^{n\times n}$. Recall that
$M=M^{T}>0$ and $K=K^{T}>0$. The matrices
$D_{ij}^{T}=D_{ij}=D_{ji}\geq 0$ represent the dissipative coupling
(due, e.g., to the dampers in the array of
Fig.~\ref{fig:pendcoupled} or to the resistors in the array of
Fig.~\ref{fig:LCcoupled}) between the $i$th and $j$th oscillators.
The matrices $R_{ij}^{T}=R_{ij}=R_{ji}\geq 0$ represent the
restorative coupling (due, e.g., to the springs in the array of
Fig.~\ref{fig:pendcoupled} or to the inductors in the array of
Fig.~\ref{fig:LCcoupled}) between the $i$th and $j$th oscillators.
(We take $D_{ii}=0$ and $R_{ii}=0$.) Let
$\sigma_{1},\,\sigma_{2},\,\ldots,\,\sigma_{n}$ be the roots of the
polynomial $d(s)=\det(sM-K)$, i.e., the eigenvalues of $K$ with
respect to $M$. Note that these $\sigma_{k}$ are also the
eigenvalues of the matrix $P:=M^{-1/2}KM^{-1/2}$. Hence
$\sigma_{k}>0$ for all $k$ because $P=P^{T}>0$. Our analysis will
assume that these eigenvalues are distinct:
$\sigma_{k}\neq\sigma_{\ell}$ for $k\neq \ell$. Under this
assumption we here intend to arrive at conditions on the set of
parameters $(M,\,K,\,(D_{ij})_{i,j=1}^{q},\,(R_{ij})_{i,j=1}^{q})$
under which the array~\eqref{eqn:array} {\em synchronizes}, i.e.,
$\|x_{i}(t)-x_{j}(t)\|\to 0$ as $t\to\infty$ for all indices $i,\,j$
and all initial conditions
$x_{1}(0),\,x_{2}(0),\,\ldots,\,x_{q}(0)$.

The identity matrix is denoted by $I_{q}\in\Real^{q\times q}$. We
let $\one_{q}\in\Real^{q}$ denote the unit vector with identical
positive entries, i.e., $\one_{q} = [1\ 1\ \cdots\ 1]^{T}/\sqrt{q}$.
Given $X\in\Complex^{n\times n}$, we let $\lambda_{k}(X)$ denote the
$k$th smallest eigenvalue of $X$ with respect to the real part. That
is, ${\rm Re}\,\lambda_{1}(X)\leq{\rm
Re}\,\lambda_{2}(X)\leq\cdots\leq{\rm Re}\,\lambda_{n}(X)$. The
2-norm of a vector $v\in\Complex^{n}$ is denoted by $\|v\|$. Recall
that $\|v\|^{2}=v^{*}v$, where $v^{*}$ denotes the conjugate
transpose of $v$. Likewise, $\|X\|$ denotes the induced 2-norm of
the matrix $X$. Let $\L(q,\,n)\subset\Real^{qn\times qn}$ denote the
set of {\em Laplacian} matrices such that each $L\in\L(q,\,n)$ has
the following structure
\begin{eqnarray*}
L =
\left[\begin{array}{cccc}\sum_{j}W_{1j}&-W_{12}&\cdots&-W_{1q}\\
-W_{21}&\sum_{j}W_{2j}&\cdots&-W_{2q}\\
\vdots&\vdots&\ddots&\vdots\\
-W_{q1}&-W_{q2}&\cdots&\sum_{j}W_{qj}
\end{array}\right]=:{\rm lap}\,(W_{ij})_{i,j=1}^{q}
\end{eqnarray*}
where the {\em weights} $W_{ij}\in\Real^{n\times n}$ satisfy
$W_{ij}^{T}=W_{ij}=W_{ji}\geq 0$ with $W_{ii}=0$. Observe the
symmetry $L=L^{T}$ and the positive semidefiniteness
$\ex^{*}L\ex=\sum_{j>i}(x_{i}-x_{j})^{*}W_{ij}(x_{i}-x_{j})\geq 0$,
where $\ex=[x_{1}^{T}\ x_{2}^{T}\ \cdots\
x_{q}^{T}]^{T}\in(\Complex^{n})^{q}$. Also ${\rm null}\,L\supset{\rm
range}\,[\one_{q}\otimes I_{n}]$, where $\otimes$ is the Kronecker
product symbol.

All the positive (semi)definite matrices we consider in this paper
will be (real and) symmetric. Therefore henceforth we write $X>0$
($X\geq 0$) to mean $X^{T}=X>0$ ($X^{T}=X\geq 0$). A simple fact
from linear algebra that we frequently use in our analysis is
\begin{eqnarray*}
X\geq 0\ \ \mbox{and}\ \ \xi^{*}X\xi=0\implies X\xi=0
\end{eqnarray*}
where $\xi$ is a vector of appropriate size. Another fact that will
receive frequent visits is the following.
\begin{fact}\label{fact:the}
Let both $X,\,Y\in\Real^{n\times n}$ be symmetric positive
semidefinite. Then ${\rm Re}\,\lambda_{k}(X+jY)\geq 0$ for all $k$.
\end{fact}

\begin{proof}
Let $\lambda\in\Complex$ be an eigenvalue of $X+jY$ and
$\xi\in\Complex^{n}$ the corresponding unit eigenvector. We can
write
\begin{eqnarray}\label{eqn:lambda}
\lambda=\xi^{*}(\lambda\xi)=\xi^{*}(X+jY)\xi=\xi^{*}X\xi+j\xi^{*}Y\xi
\end{eqnarray}
which yields ${\rm Re}\,\lambda=\xi^{*}X\xi\geq0$ because
$X,\,Y\geq0$. The fact follows since $\lambda$ was arbitrary.
\end{proof}

\section{Steady state solutions}\label{sec:SS}

Consider the array of coupled pendulums shown in
Fig.~\ref{fig:pendcoupled} under arbitrary initial conditions.
Devoid of any external interference, this assembly is unable to
generate mechanical energy. Moreover, some of its initial energy
will be gradually lost through the dampers as heat. The outcome is
that in the long run the array has to settle into a constant energy
state, the {\em steady state}. One way to show that the array
synchronizes (if it does) therefore would be to establish that no
steady state solution admits asynchronous oscillations. This is the
approach we adopt for our analysis in this section.

Let us employ the coordinate change $z_{i}:=M^{1/2}x_{i}$ for
$i=1,\,2,\,\ldots,\,q$. In the new coordinates, the
array~\eqref{eqn:array} takes the form
\begin{eqnarray}\label{eqn:array2}
{\ddot z}_{i}+Pz_{i}+\sum_{j=1}^{q}M^{-1/2}D_{ij}M^{-1/2}({\dot
z}_{i}-{\dot z}_{j})+\sum_{j=1}^{q}
M^{-1/2}R_{ij}M^{-1/2}(z_{i}-z_{j})=0\,,\qquad
i=1,\,2,\,\ldots,\,q\,.
\end{eqnarray}
Recall that $P=M^{-1/2}KM^{-1/2}$ whose eigenvalues
$\sigma_{1},\,\sigma_{2},\,\ldots,\,\sigma_{n}$ are distinct and
positive. Let $\zi=[z_{1}^{T}\ z_{2}^{T}\ \cdots\ z_{q}^{T}]^{T}$
and the matrices $L_{\rm d},\,L_{\rm r}\in\setL(q,\,n)$ be
constructed as
\begin{eqnarray*}
L_{\rm d}&:=&{\rm lap}\,(M^{-1/2}D_{ij}M^{-1/2})_{i,j=1}^{q}\,,\\
L_{\rm r}&:=&{\rm lap}\,(M^{-1/2}R_{ij}M^{-1/2})_{i,j=1}^{q}\,.
\end{eqnarray*}
These Laplacian matrices allow us to express \eqref{eqn:array2} as
\begin{eqnarray}\label{eqn:bigz}
{\ddot\zi}+[I_{q}\otimes P]\zi+L_{\rm d}{\dot\zi}+L_{\rm r}\zi=0\,.
\end{eqnarray}
Note that the array~\eqref{eqn:array} synchronizes (only) when the
array~\eqref{eqn:array2} does. And the synchronization of the
array~\eqref{eqn:array2} is equivalent to that every solution
$\zi(t)$ of \eqref{eqn:bigz} converges to the subspace ${\rm
range}\,[\one_{q}\otimes I_{n}]$. Consider now the Lyapunov function
\begin{eqnarray*}
W(\zi,\,\dot\zi)=\frac{1}{2}\zi^{T}\left([I_{q}\otimes P]+L_{\rm
r}\right)\zi+\frac{1}{2}{\dot\zi}^{T}{\dot\zi}
\end{eqnarray*}
which is positive definite since $[I_{q}\otimes P]>0$ and $L_{\rm
r}\geq 0$ imply $[I_{q}\otimes P]+L_{\rm r}>0$. The time derivative
of this function along the solutions of \eqref{eqn:bigz} reads
\begin{eqnarray*}
\frac{d}{dt}W(\zi(t),\,\dot\zi(t)) = -\dot\zi(t)^{T}L_{\rm
d}\dot\zi(t)\,.
\end{eqnarray*}
Note that the righthand side is negative semidefinite since $L_{\rm
d}\geq 0$. Hence by Lyapunov stability theorem each pair
$(\zi(t),\,\dot\zi(t))$ is bounded and by Krasovskii-LaSalle
principle \cite{khalil96}, every solution converges to some region
contained in the set $\{(\zi,\,\dot\zi):\dot W(\zi,\,\dot\zi)=0\}$.
In other words, every steady state solution $\zi_{\rm ss}(t)$ of
\eqref{eqn:bigz} should identically satisfy $\dot\zi_{\rm
ss}(t)^{T}L_{\rm d}\dot\zi_{\rm ss}(t)=0$, which (thanks to $L_{\rm
d}\geq 0$) is equivalent to
\begin{eqnarray}\label{eqn:Ldz0}
L_{\rm d}\dot\zi_{\rm ss}(t)\equiv 0\,.
\end{eqnarray}
Combining $\eqref{eqn:bigz}$ and \eqref{eqn:Ldz0} at once yields
\begin{eqnarray}\label{eqn:bigzss}
{\ddot\zi}_{\rm ss}+([I_{q}\otimes P]+L_{\rm r})\zi_{\rm ss}=0\,.
\end{eqnarray}
Let $p\leq qn$ be the number of distinct eigenvalues of
$[I_{q}\otimes P]+L_{\rm r}$  and
$\rho_{1},\,\rho_{2},\,\ldots,\,\rho_{p}>0$ denote these
eigenvalues. Note that $\rho_{k}>0$ because $[I_{q}\otimes P]+L_{\rm
r}>0$. Now, the solution to \eqref{eqn:bigzss} has the form \cite[\S
23]{arnold89}
\begin{eqnarray}\label{eqn:arnold}
\zi_{\rm ss}(t)={\rm Re}\sum_{k=1}^{p}e^{j\omega_{k}t}\xi_{k}
\end{eqnarray}
where $\omega_{k}=\sqrt{\rho_{k}}$ are distinct and positive, and
each $\xi_{k}\in(\Complex^{n})^{q}$ (some of which may be zero)
satisfies
\begin{eqnarray}\label{eqn:arnold2}
([I_{q}\otimes P]+L_{\rm r}-\omega_{k}^{2}I_{qn})\xi_{k}=0\,.
\end{eqnarray}
Note that the \eqref{eqn:Ldz0} and \eqref{eqn:arnold} imply
\begin{eqnarray}\label{eqn:arnold3}
L_{\rm d}\xi_{k}=0
\end{eqnarray}
since $\omega_{k}$ are distinct and nonzero. Combining
\eqref{eqn:arnold2} and \eqref{eqn:arnold3} we can write
\begin{eqnarray}\label{eqn:arnold4}
\xi_{k}\in{\rm null}\left[\begin{array}{c}[I_{q}\otimes P]+L_{\rm r}-\omega_{k}^{2} I_{qn}\\
L_{\rm d}\end{array}\right]\,.
\end{eqnarray}
Suppose now the following (PBH test like) condition holds
\begin{eqnarray}\label{eqn:PBH}
{\rm null}\left[\begin{array}{c}[I_{q}\otimes P]+L_{\rm r}-\lambda I_{qn}\\
L_{\rm d}\end{array}\right]\subset {\rm range}\,[\one_{q}\otimes
I_{n}]\ \ \mbox{for all}\ \ \lambda\in\Real\,.
\end{eqnarray}
Then \eqref{eqn:arnold4} implies $\xi_{k}\in{\rm range}\,[{\bf
1}_{q}\otimes I_{n}]$ for all $k$. By \eqref{eqn:arnold} this
readily yields $\zi_{\rm ss}(t)\in{\rm range}\,[{\bf 1}_{q}\otimes
I_{n}]$ for all $t$. Therefore \eqref{eqn:PBH} is sufficient for the
array~\eqref{eqn:array} to synchronize.

Let us also investigate the necessity. We begin by supposing that
the condition~\eqref{eqn:PBH} fails to hold. Then we can find an
eigenvalue $\rho_{k}=\omega_{k}^{2}$ and an eigenvector
$\xi\in(\Real^{n})^{q}$ satisfying $\xi\notin{\rm range}\,
[\one_{q}\otimes I_{n}]$, $L_{\rm d}\xi=0$, and $([I_{q}\otimes
P]+L_{\rm r}-\omega_{k}^{2} I_{qn})\xi=0$. Using the pair
$(\omega_{k},\,\xi)$ let us construct the function
$\zeta:\Real\to(\Real^{n})^{q}$ as $\zeta(t)={\rm
Re}\,(e^{j\omega_{k} t}\xi)$. This function satisfies the following
properties. First, since $\xi\notin{\rm range}\,[\one_{q}\otimes
I_{n}]$, we have
\begin{eqnarray}\label{eqn:vball1}
\zeta(t)=\xi\notin{\rm range}\,[{\bf 1}_{q}\otimes I_{n}]\ \
\mbox{for}\ \ t=0,\,T,\,2T,\,\ldots
\end{eqnarray}
where $T=2\pi/\omega_{k}$. Second, since $L_{\rm d}\xi=0$, we have
at all times
\begin{eqnarray}\label{eqn:vball2}
L_{\rm d}\dot{\zeta}(t)={\rm Re}\,(j\omega_{k}e^{j\omega_{k}
t}L_{\rm d}\xi)=0\,.
\end{eqnarray}
Third, since $([I_{q}\otimes P]+L_{\rm r}-\omega_{k}^{2}
I_{qn})\xi=0$, we can write at all times
\begin{eqnarray*}
{\ddot \zeta}(t)+([I_{q}\otimes P]+L_{\rm
r})\zeta(t)&=&-\omega_{k}^{2}\zeta(t)+([I_{q}\otimes P]+L_{\rm r})\zeta(t)\\
&=&([I_{q}\otimes P]+L_{\rm r}-\omega_{k}^{2}I_{qn}){\rm Re}\,(e^{j\omega_{k} t}\xi)\\
&=&{\rm Re}\,(e^{j\omega_{k} t}([I_{q}\otimes P]+L_{\rm r}-\omega_{k}^{2}I_{qn})\xi)\\
&=&0
\end{eqnarray*}
which together with \eqref{eqn:vball2} leads to
\begin{eqnarray*}
{\ddot \zeta}(t)+[I_{q}\otimes P]\zeta(t)+L_{\rm
d}\dot{\zeta}(t)+L_{\rm r}\zeta(t)\equiv0\,.
\end{eqnarray*}
Hence $\zeta(t)$ is a valid solution of \eqref{eqn:bigz}. But it is
clear from \eqref{eqn:vball1} that $\zeta(t)$ does not converge to
${\rm range}\,[{\bf 1}_{q}\otimes I_{n}]$. This means that the
condition~\eqref{eqn:PBH} is not only sufficient but also necessary
for the synchronization of the array~\eqref{eqn:array}. We have
therefore established:

\begin{lemma}\label{lem:PBH}
The array~\eqref{eqn:array} synchronizes if and only if
\eqref{eqn:PBH} holds.
\end{lemma}

We now convert the condition~\eqref{eqn:PBH} to another form, which
will prove more suitable for later analysis. To this end, we
construct the complex matrix
\begin{eqnarray*}
\Gamma:=L_{\rm d}+j([I_{q}\otimes P]+L_{\rm r})\,.
\end{eqnarray*}
A few observations on the spectrum of $\Gamma$ are in order. Note
that $L_{\rm d}[\one_{q}\otimes v]=0$ for all $v\in\Complex^{n}$
thanks to $L_{\rm d}\in\setL(q,\,n)$. The same goes for $L_{\rm r}$.
Letting $v_{1},\,v_{2},\,\ldots,\,v_{n}\in\Real^{n}$ be the
(linearly independent) unit eigenvectors of $P$ corresponding to the
eigenvalues $\sigma_{1},\,\sigma_{2},\,\ldots,\,\sigma_{n}$,
respectively, we can thus write for $k=1,\,2,\,\ldots,\,n$
\begin{eqnarray*}
\Gamma [\one_{q}\otimes v_{k}] &=& L_{\rm d}[\one_{q}\otimes
v_{k}]+j([I_{q}\otimes P][\one_{q}\otimes v_{k}]+L_{\rm
r}[\one_{q}\otimes v_{k}])\\
&=& j[I_{q}\otimes P][\one_{q}\otimes v_{k}]\\
&=& j[(I_{q}\one_{q})\otimes (Pv_{k})]\\
&=& j\sigma_{k}[\one_{q}\otimes v_{k}]\,.
\end{eqnarray*}
Therefore each $j\sigma_{k}$ is an eigenvalue of $\Gamma$ with the
corresponding eigenvector $[\one_{q}\otimes v_{k}]$. Since
$\sigma_{k}\neq\sigma_{\ell}$ for $k\neq \ell$ and the open left
half plane contains no eigenvalue of $\Gamma$ by
Fact~\ref{fact:the}; we can list, without loss of generality, the
first $n$ eigenvalues as $\lambda_{k}(\Gamma)=j\sigma_{k}$ for
$k=1,\,2,\,\ldots,\,n$. It turns out that the next eigenvalue in
line is closely related to synchronization:

\begin{lemma}\label{lem:lambda2}
The condition~\eqref{eqn:PBH} holds if and only if ${\rm
Re}\,\lambda_{n+1}(\Gamma)>0$.
\end{lemma}

\begin{proof}
Suppose ${\rm Re}\,\lambda_{n+1}(\Gamma)\leq 0$. This implies ${\rm
Re}\,\lambda_{n+1}(\Gamma)=0$ because $\Gamma$ can have no
eigenvalue with negative real part. Let therefore
$\lambda_{n+1}(\Gamma)=j\beta$ with $\beta\in\Real$. There are two
possibilities. Either (i) $j\beta=j\sigma_{k}$ for some $k$ or (ii)
not. Consider the case (i). Without loss of generality let us take
$j\beta=j\sigma_{1}$. That is, the eigenvalue $j\sigma_{1}$ is
repeated. Then there should be at least two linearly independent
eigenvectors of $\Gamma$ corresponding to the eigenvalue
$j\sigma_{1}$. To see this suppose otherwise. Then $[\one_{q}\otimes
v_{1}]$ would be the only (unit) eigenvector associated to the
eigenvalue $j\sigma_{1}$ and there would have to exist a generalized
eigenvector $\xi_{1}\in(\Complex^{n})^{q}$ satisfying
$(\Gamma-j\sigma_{1}I_{qn})\xi_{1}=[\one_{q}\otimes v_{1}]$. This
however would lead to the following contradiction
\begin{eqnarray*}
1 =[\one_{q}\otimes v_{1}]^{T}[\one_{q}\otimes v_{1}]
=[\one_{q}\otimes v_{1}]^{T}(\Gamma-j\sigma_{1}I_{qn})\xi_{1}
=((\Gamma-j\sigma_{1}I_{qn})[\one_{q}\otimes v_{1}])^{T}\xi_{1} =0
\end{eqnarray*}
since $\Gamma^{T}=\Gamma$. Therefore we can find an eigenvector
$\xi\in(\Complex^{n})^{q}$ corresponding to the eigenvalue
$j\sigma_{1}$ satisfying $\xi\notin{\rm span}\,\{[\one_{q}\otimes
v_{1}]\}$. Then it follows that
\begin{eqnarray}\label{eqn:xi}
\xi\notin{\rm span}\,\{[\one_{q}\otimes v_{1}],\,[\one_{q}\otimes
v_{2}],\,\ldots,\,[\one_{q}\otimes v_{n}]\}={\rm
range}\,[\one_{q}\otimes I_{n}]\,.
\end{eqnarray}
As for the case (ii), i.e., $j\beta\neq j\sigma_{k}$ for all $k$, it
is clear that an eigenvector of $j\beta$, call it $\xi$, should
again satisfy \eqref{eqn:xi}. To sum up, whenever ${\rm
Re}\,\lambda_{n+1}(\Gamma)\leq 0$, there exists a nonzero vector
$\xi\in(\Complex^{n})^{q}$ and a real number $\beta$ satisfying
$\Gamma\xi=j\beta\xi$ and \eqref{eqn:xi}. Without loss of generality
let $\|\xi\|=1$. Then \eqref{eqn:lambda} allows us to write
\begin{eqnarray*}
j\beta = \xi^{*}L_{\rm d}\xi+j\xi^{*}([I_{q}\otimes P]+L_{\rm r})\xi
\end{eqnarray*}
which implies $L_{\rm d}\xi=0$. Then we can write
$j\beta\xi=\Gamma\xi=j([I_{q}\otimes P]+L_{\rm r})\xi$, yielding
$([I_{q}\otimes P]+L_{\rm r}-\beta I_{qn})\xi=0$. Hence $\xi$
satisfies
\begin{eqnarray}\label{eqn:xinull}
\xi\in{\rm null}\left[\begin{array}{c}[I_{q}\otimes P]+L_{\rm r}-\beta I_{qn}\\
L_{\rm d}\end{array}\right]\,.
\end{eqnarray}
Finally, combining \eqref{eqn:xi} and \eqref{eqn:xinull} yields that
the condition~\eqref{eqn:PBH} cannot be true.

Now we show the other direction. Suppose \eqref{eqn:PBH} is not
true. Then we can find $\beta\in\Real$ and a vector $\xi\notin{\rm
range}\,[\one_{q}\otimes I_{n}]$ that satisfy $L_{\rm d}\xi=0$ and
$([I_{q}\otimes P]+L_{\rm r})\xi=\beta\xi$. This yields
$\Gamma\xi=j\beta\xi$. By \eqref{eqn:xi} we see that $\xi$ lies
outside the subspace spanned by the linearly independent
eigenvectors $[\one_{q}\otimes v_{1}],\,[\one_{q}\otimes
v_{2}],\,\ldots,\,[\one_{q}\otimes v_{n}]$ of $\Gamma$. Recall that
the eigenvalues associated to these eigenvectors are
$j\sigma_{1},\,j\sigma_{2},\,\ldots,\,j\sigma_{n}$. Hence, together
with $\xi$, there are at least $n+1$ linearly independent
eigenvectors whose eigenvalues lie on the imaginary axis. This
implies ${\rm Re}\,\lambda_{n+1}(\Gamma)$ cannot be strictly
positive.
\end{proof}

\vspace{0.12in}

Lemma~\ref{lem:PBH} and Lemma~\ref{lem:lambda2} yield:

\begin{theorem}\label{thm:main}
The array~\eqref{eqn:array} synchronizes if and only if ${\rm
Re}\,\lambda_{n+1}(\Gamma)>0$.
\end{theorem}

To develop some insight on Theorem~\ref{thm:main} we bring up some
of its consequences concerning a number of special yet important
cases. We first regenerate some known results on harmonic
oscillators; then (in the following sections) we proceed to novel
implications. Synchronization of coupled harmonic oscillators (i.e.,
the array~\eqref{eqn:array} under $n=1$) is a thoroughly
investigated problem; see, for instance,
\cite{ren08,su09,zhou12,sun15}. Many interesting results have
appeared recently, each of which studies a certain generalization of
the nominal setup: an array of identical oscillators (e.g., 1-link
pendulums) coupled only by dissipative components (e.g., dampers).
In this simplest case synchronization is easy to understand. It is
intuitively clear that if a pair of pendulums are connected by a
damper then their motions have to have synchronized in the steady
state. Consequently, the entire array synchronizes if its
interconnection {\em graph} (where each node represents an
oscillator and each edge a damper) is connected. This well-known,
fundamental result makes the first corollary of
Theorem~\ref{thm:main} since the algebraic condition for a graph to
be connected is that its Laplacian has a simple eigenvalue at the
origin, i.e., its second smallest eigenvalue (also known as Fiedler
eigenvalue) is positive.

\begin{corollary}\label{cor:d}
Suppose $n=1$ and $R_{ij}=0$ for all $i,\,j$. Then the
array~\eqref{eqn:array} synchronizes if and only if
$\lambda_{2}(L_{\rm d})>0$.
\end{corollary}

\begin{proof}
That $n=1$ renders the matrix $P$ a real scalar. In particular,
$P=\sigma_{1}$. We can therefore write
\begin{eqnarray}\label{eqn:ruedes}
{\rm Re}\,\lambda_{n+1}(\Gamma)
&=& {\rm Re}\,\lambda_{n+1}(L_{\rm d}+j([I_{q}\otimes P]+L_{\rm r}))\nonumber\\
&=& {\rm Re}\,\lambda_{2}(L_{\rm d}+j\sigma_{1}I_{q}+jL_{\rm r})\nonumber\\
&=& {\rm Re}\,(\lambda_{2}(L_{\rm d}+jL_{\rm r})+j\sigma_{1})\nonumber\\
&=& {\rm Re}\,\lambda_{2}(L_{\rm d}+jL_{\rm r})\,.
\end{eqnarray}
Now, $R_{ij}=0$ yields $L_{\rm r}=0$. Also, $L_{\rm d}\geq 0$
implies that all the eigenvalues of the Laplacian $L_{\rm d}$ are
real. Hence
\begin{eqnarray}\label{eqn:anges}
{\rm Re}\,\lambda_{2}(L_{\rm d}+jL_{\rm r})=\lambda_{2}(L_{\rm
d})\,.
\end{eqnarray}
The result follows by \eqref{eqn:ruedes}, \eqref{eqn:anges}, and
Theorem~\ref{thm:main}.
\end{proof}

\vspace{0.12in}

Corollary~\ref{cor:d} has the following generalization covering the
case where the 1-link pendulums are coupled by not only dampers, but
also springs \cite{tuna17}.

\begin{corollary}\label{cor:dr}
Suppose $n=1$. Then the array~\eqref{eqn:array} synchronizes if and
only if ${\rm Re}\,\lambda_{2}(L_{\rm d}+jL_{\rm r})>0$.
\end{corollary}

\begin{proof}
Combine \eqref{eqn:ruedes} and Theorem~\ref{thm:main}.
\end{proof}

\section{Weak restorative coupling}\label{sec:WR}

In this section we study the synchronization of small oscillations
under weak restorative coupling. To investigate how the strength of
restorative coupling effects synchronization let us replace $R_{ij}$
in \eqref{eqn:array} with $\varepsilon R_{ij}$, yielding the
dynamics
\begin{eqnarray}\label{eqn:arrayeps}
M{\ddot x}_{i}+Kx_{i}+\sum_{j=1}^{q}D_{ij}({\dot x}_{i}-{\dot
x}_{j})+\varepsilon\sum_{j=1}^{q}R_{ij}(x_{i}-x_{j})=0\,,\qquad
i=1,\,2,\,\ldots,\,q
\end{eqnarray}
where the scalar $\varepsilon>0$ represents the coupling strength.
Our assumptions on the matrices $M,\,K,\,D_{ij},\,R_{ij}$ are same
as before. A slight addition, however, is that we assume throughout
this section that not all $R_{ij}$ are zero, i.e., $R_{ij}\neq0$ for
at least one pair $(i,\,j)$. The case where there is no restorative
coupling (i.e., all $R_{ij}=0$) is studied in the next section. For
our new array~\eqref{eqn:arrayeps} let us define
\begin{eqnarray*}
\Gamma_{\varepsilon}:=L_{\rm d}+j([I_{q}\otimes P]+\varepsilon
L_{\rm r})\,.
\end{eqnarray*}
We infer from Theorem~\ref{thm:main} that the
array~\eqref{eqn:arrayeps} synchronizes if and only if ${\rm
Re}\,\lambda_{n+1}(\Gamma_{\varepsilon})>0$. Recall that
$v_{1},\,v_{2},\,\ldots,\,v_{n}\in\Real^{n}$ denote the (linearly
independent) unit eigenvectors of $P$ corresponding to the distinct
eigenvalues $\sigma_{1},\,\sigma_{2},\,\ldots,\,\sigma_{n}$,
respectively. Since $P$ is real and symmetric the matrix $V=[v_{1}\
v_{2}\ \cdots\ v_{n}]$ is orthogonal, i.e., $V^{T}V=I_{n}$. Let
$\Lambda=\diag(\sigma_{1},\,\sigma_{2},\,\ldots,\,\sigma_{n})$. Note
that $\Lambda=V^{T}PV$. Let us now construct the matrices
$G,\,B\in\Real^{qn\times qn}$ as
\begin{eqnarray*}
G = \left[
\begin{array}{cccc}
G_{11}&G_{12}&\cdots&G_{1n}\\
G_{21}&G_{22}&\cdots&G_{2n}\\
\vdots&\vdots&\ddots&\vdots\\
G_{n1}&G_{n2}&\cdots&G_{nn}
\end{array}
\right]\ ,\qquad B = \left[
\begin{array}{cccc}
B_{11}&B_{12}&\cdots&B_{1n}\\
B_{21}&B_{22}&\cdots&B_{2n}\\
\vdots&\vdots&\ddots&\vdots\\
B_{n1}&B_{n2}&\cdots&B_{nn}
\end{array}
\right]
\end{eqnarray*}
where $G_{k\ell}=[I_{q}\otimes v_{k}^{T}]L_{\rm d}[I_{q}\otimes
v_{\ell}]\in\Real^{q\times q}$ and $B_{k\ell}=[I_{q}\otimes
v_{k}^{T}]L_{\rm r}[I_{q}\otimes v_{\ell}]\in\Real^{q\times q}$ for
$k,\,\ell=1,\,2,\,\ldots,\,n$.

\begin{lemma}\label{lem:Gkk}
The matrices $G_{kk},\,B_{kk}$ are Laplacian, i.e.,
$G_{kk},\,B_{kk}\in\setL(q,\,1)$ for all $k=1,\,2,\,\ldots,\,n$.
\end{lemma}

\begin{proof}
We can write
\begin{eqnarray}\label{eqn:Gkklap}
G_{kk}
&=&[I_{q}\otimes v_{k}^{T}]L_{\rm d}[I_{q}\otimes v_{k}]\nonumber\\
&=&[I_{q}\otimes v_{k}^{T}]({\rm lap}\,(M^{-1/2}D_{ij}M^{-1/2})_{i,j=1}^{q})[I_{q}\otimes v_{k}]\nonumber\\
&=&{\rm lap}\,(v_{k}^{T}M^{-1/2}D_{ij}M^{-1/2}v_{k})_{i,j=1}^{q}\,.
\end{eqnarray}
Likewise, $B_{kk}={\rm
lap}\,(v_{k}^{T}M^{-1/2}R_{ij}M^{-1/2}v_{k})_{i,j=1}^{q}$.
\end{proof}

\vspace{0.12in}

It is not difficult to see that the matrices $G,\,B$ satisfy
\begin{subeqnarray*}
G&=&\Pi^{T}[I_{q}\otimes V^{T}]L_{\rm d}[I_{q}\otimes V]\Pi\\
B&=&\Pi^{T}[I_{q}\otimes V^{T}]L_{\rm r}[I_{q}\otimes V]\Pi
\end{subeqnarray*}
where $\Pi\in\Real^{qn\times qn}$ is the permutation matrix that
yields $\Pi^{T}[X\otimes Y]\Pi=Y\otimes X$ for all
$X\in\Real^{q\times q}$ and $Y\in\Real^{n\times n}$. Hence
$G,\,B\geq 0$. Define
\begin{eqnarray*}
\Omega_{\varepsilon}:=G+j([\Lambda\otimes I_{q}]+\varepsilon B)\,.
\end{eqnarray*}
Note that $\Omega_{\varepsilon}=\Pi^{T}[I_{q}\otimes
V^{T}]\Gamma_{\varepsilon}[I_{q}\otimes V]\Pi$. That is,
$\Omega_{\varepsilon}$ and $\Gamma_{\varepsilon}$ are similar
matrices. Therefore they share the same eigenvalues. Since the
array~\eqref{eqn:arrayeps} synchronizes if and only if ${\rm
Re}\,\lambda_{n+1}(\Gamma_{\varepsilon})>0$, we have the following
result.

\begin{proposition}\label{prop:swap}
The array~\eqref{eqn:arrayeps} synchronizes if and only if ${\rm
Re}\,\lambda_{n+1}(\Omega_{\varepsilon})>0$.
\end{proposition}

\begin{remark}\label{rem:observe}
Although we assume $\varepsilon>0$ here, it is not difficult to see
that Proposition~\ref{prop:swap} still holds for the case
$\varepsilon=0$. This observation will be useful in the next section
when we consider the pure dissipative coupling scenario.
\end{remark}

Let $\{e_{1},\,e_{2},\,\ldots,\,e_{n}\}$ be the canonical basis for
$\Complex^{n}$, i.e., $e_{k}$ is the $k$th column of $I_{n}$. Note
that we have $[e_{k}\otimes \one_{q}]^{T}G[e_{k}\otimes
\one_{q}]=\one_{q}^{T}G_{kk}\one_{q}=0$ because $G_{kk}\one_{q}=0$
thanks to that $G_{kk}\in\setL(q,\,1)$ by Lemma~\ref{lem:Gkk}. Since
$G\geq 0$ this allows us to claim $G[e_{k}\otimes \one_{q}]=0$ for
all $k$. Likewise, $B[e_{k}\otimes \one_{q}]=0$. We can thus write
\begin{eqnarray*}
\Omega_{\varepsilon}[e_{k}\otimes \one_{q}] &=& (G+j([\Lambda\otimes
I_{q}]+\varepsilon B))[e_{k}\otimes \one_{q}]\\
&=& j[\Lambda\otimes I_{q}][e_{k}\otimes \one_{q}]\\
&=& j[(\Lambda e_{k})\otimes (I_{q}\one_{q})]\\
&=& j\sigma_{k}[e_{k}\otimes \one_{q}]\,.
\end{eqnarray*}
Hence each $j\sigma_{k}$ is an eigenvalue of $\Omega_{\varepsilon}$
with the corresponding eigenvector $[e_{k}\otimes \one_{q}]$. Since
by Fact~\ref{fact:the} all the eigenvalues of $\Omega_{\varepsilon}$
are on the closed right half plane, we can let, without loss of
generality, $\lambda_{k}(\Omega_{\varepsilon})=j\sigma_{k}$ for
$k=1,\,2,\,\ldots,\,n$. Define the positive numbers
$\bar\sigma,\,\bar\mu$ as
\begin{eqnarray*}
\bar\sigma&:=&\frac{1}{2}\min_{k\neq\ell}|\sigma_{k}-\sigma_{\ell}|\,,\\
\bar\mu&:=&\frac{1}{2}\min_{k}\left(\min_{\lambda_{i}(B_{kk})\neq\lambda_{j}(B_{kk})}|\lambda_{i}(B_{kk})-\lambda_{j}(B_{kk})|\right)\,.
\end{eqnarray*}

\begin{lemma}\label{lem:xi}
Let $\xi\in(\Complex^{q})^{n}$ be a unit vector satisfying
$\Omega_{\varepsilon}\xi=j\beta\xi$ for some $\beta\in\Real$. There
exist an index $k\in\{1,\,2,\,\ldots,\,n\}$ and an eigenvector
$w\in\Complex^{q}$ of $B_{kk}$ such that
\begin{eqnarray}\label{eqn:ceps}
\|\xi-[e_{k}\otimes w]\|\leq
\left[\frac{\sqrt{n-1}\|B\|}{\bar\sigma}\left(1+
\frac{\|B\|}{\bar\mu}\right)\right]\varepsilon\,.
\end{eqnarray}
\end{lemma}

\begin{proof}
Let $\xi$ be a unit vector satisfying
$\Omega_{\varepsilon}\xi=j\beta\xi$. We have
$j\beta=\xi^{*}G\xi+j\xi^{*}([\Lambda\otimes I_{q}]+\varepsilon
B)\xi$ by \eqref{eqn:lambda}. Since $G\geq 0$ and $[\Lambda\otimes
I_{q}]+\varepsilon B>0$ we have to have $\xi^{*}G\xi=0$ which in
turn implies $G\xi=0$. Thence $\Omega_{\varepsilon}\xi=j\beta\xi$
yields
\begin{eqnarray}\label{eqn:prickball}
([\Lambda\otimes I_{q}]+\varepsilon B)\xi=\beta\xi\,.
\end{eqnarray}
We have by \cite[Cor.~8.1.6]{golub96} for all
$i=1,\,2,\,\ldots,\,qn$
\begin{eqnarray*}
|\lambda_{i}([\Lambda\otimes I_{q}]+\varepsilon
B)-\lambda_{i}([\Lambda\otimes I_{q}])|\leq \|B\|\varepsilon\,.
\end{eqnarray*}
Since $\lambda_{i}([\Lambda\otimes
I_{q}])\in\{\sigma_{1},\,\sigma_{2},\,\ldots,\,\sigma_{n}\}$, we
must have $\beta=\sigma_{k}+h$ for some $k\in\{1,\,2,\,\ldots,\,n\}$
and $|h|\leq\|B\|\varepsilon$. Without loss of generality let
$\beta=\sigma_{1}+h$. Let $\xi$ be partitioned as $\xi=[u_{1}^{T}\
u_{2}^{T}\ \cdots\ u_{n}^{T}]^{T}$ with $u_{k}\in\Complex^{q}$.
Since $\|\xi\|=1$ we have $\sum_{\ell=1}^{n}\|u_{\ell}\|^{2}=1$. Let
us now rewrite \eqref{eqn:prickball} as
\begin{eqnarray*}
\left[\begin{array}{c}\sigma_{1}u_{1}\\
\sigma_{2}u_{2}\\ \vdots\\ \sigma_{n}u_{n}\end{array}\right]+
\varepsilon B\left[\begin{array}{c}u_{1}\\
u_{2}\\ \vdots\\ u_{n}\end{array}\right]=(\sigma_{1}+h)\left[\begin{array}{c}u_{1}\\
u_{2}\\ \vdots\\ u_{n}\end{array}\right]
\end{eqnarray*}
which we decompose into $n$ equations, the first of which is
\begin{eqnarray}\label{eqn:prickball2}
\left(B_{11}-\frac{h}{\varepsilon}
I_{q}\right)u_{1}=-\sum_{\ell=2}^{n}B_{1\ell}u_{\ell}
\end{eqnarray}
and the remaining $n-1$ are
\begin{eqnarray}\label{eqn:prickball3}
(\sigma_{k}-\sigma_{1})u_{k}=hu_{k}-\varepsilon\sum_{\ell=1}^{n}B_{k
\ell}u_{\ell}\,,\qquad k = 2,\,3,\,\ldots,\,n\,.
\end{eqnarray}
Using \cite[Eq.~(2.3.13)]{golub96},
$\|\xi\|=\sum_{\ell=1}^{n}\|u_{\ell}\|^{2}=1$, and
$|h|\leq\|B\|\varepsilon$ we infer from \eqref{eqn:prickball3}
\begin{eqnarray}\label{eqn:prickball4}
\|u_{k}\|&\leq&\frac{1}{|\sigma_{k}-\sigma_{1}|}\left(|h|\cdot\|u_{k}\|+\varepsilon\left\|\sum_{\ell=1}^{n}B_{k\ell}u_{\ell}\right\|\right)\nonumber\\
&\leq&\frac{1}{|\sigma_{k}-\sigma_{1}|}\left(|h|+\varepsilon\|B\|\cdot\|\xi\|\right)\nonumber\\
&\leq&\frac{2\|B\|\varepsilon}{|\sigma_{k}-\sigma_{1}|}\nonumber\\
&\leq&\frac{\|B\|\varepsilon}{\bar\sigma}\,,\qquad k =
2,\,3,\,\ldots,\,n\,.
\end{eqnarray}
Let $\zeta=\xi-[e_{1}\otimes u_{1}]=[0^{T}\ u_{2}^{T}\ \cdots\
u_{n}^{T}]^{T}$, for which we have
$\|\zeta\|\leq\sqrt{n-1}\|B\|\varepsilon/\bar\sigma$ by
\eqref{eqn:prickball4}. Letting $\alpha=h/\varepsilon$ and using
\eqref{eqn:prickball2} we obtain
\begin{eqnarray}\label{eqn:prickball5}
\|(B_{11}-\alpha I_{q})u_{1}\|
&=&\left\|\sum_{\ell=2}^{n}B_{1\ell}u_{\ell}\right\|\nonumber\\
&\leq&\|B\|\cdot\|\zeta\|\nonumber\\
&\leq&\frac{\sqrt{n-1}\|B\|^{2}\varepsilon}{\bar\sigma}\,.
\end{eqnarray}
We have $B_{11}\geq 0$ by Lemma~\ref{lem:Gkk}. This means we can
find $m\leq q$ pairwise orthogonal eigenvectors
$w_{1},\,w_{2},\,\ldots,\,w_{m}\in\Complex^{q}$ with corresponding
distinct eigenvalues $\mu_{1},\,\mu_{2},\,\ldots,\,\mu_{m}\in\Real$
such that $B_{11}w_{i}=\mu_{i}w_{i}$ and
$u_{1}=w_{1}+w_{2}+\cdots+w_{m}$. Using the pairwise orthogonality
of the vectors $w_{i}$ we can write
\begin{eqnarray}\label{eqn:prickball6}
\left(\sum_{i=1}^{m}|\mu_{i}-\alpha|^{2}\|w_{i}\|^{2}\right)^{1/2}
&=& \left\|\sum_{i=1}^{m}(\mu_{i}-\alpha)w_{i}\right\|\nonumber\\
&=& \left\|\sum_{i=1}^{m}(B_{11}-\alpha I_{q})w_{i}\right\|\nonumber\\
&=& \left\|(B_{11}-\alpha I_{q})\sum_{i=1}^{m}w_{i}\right\|\nonumber\\
&=& \left\|(B_{11}-\alpha I_{q})u_{1}\right\|\,.
\end{eqnarray}
Without loss of generality suppose $|\mu_{1}-\alpha|\leq
|\mu_{i}-\alpha|$ for $i=2,\,3,\,\ldots,\,m$. Note then that
$|\mu_{i}-\alpha|\geq \bar\mu$ for $i=2,\,3,\,\ldots,\,m$. Hence we
can write by \eqref{eqn:prickball5} and \eqref{eqn:prickball6}
\begin{eqnarray}\label{eqn:prickball7}
\|u_{1}-w_{1}\|
&=&\left(\sum_{i=2}^{m}\|w_{i}\|^2\right)^{1/2}\nonumber\\
&\leq&\frac{1}{{\bar\mu}}\left(\sum_{i=2}^{m}|\mu_{i}-\alpha|^{2}\|w_{i}\|^2\right)^{1/2}\nonumber\\
&\leq&\frac{\sqrt{n-1}\|B\|^{2}\varepsilon}{\bar\mu\bar\sigma}\,.
\end{eqnarray}
Recall $\|\zeta\|\leq\sqrt{n-1}\|B\|\varepsilon/\bar\sigma$. Hence
\eqref{eqn:prickball7} yields
\begin{eqnarray*}
\|\xi-[e_{1}\otimes w_{1}]\|
&=& \|[e_{1}\otimes u_{1}]+\zeta-[e_{1}\otimes w_{1}]\| \\
&=& \|[e_{1}\otimes (u_{1}-w_{1})]+\zeta \|\\
&=& \|e_{1}\otimes (u_{1}-w_{1})\|+\|\zeta \|\\
&=& \|u_{1}-w_{1}\|+\|\zeta\|\\
&\leq&
\frac{\sqrt{n-1}\|B\|^{2}\varepsilon}{\bar\mu\bar\sigma}+\frac{\sqrt{n-1}\|B\|\varepsilon}{\bar\sigma}\\
&=&\left[\frac{\sqrt{n-1}\|B\|}{\bar\sigma}\left(1+
\frac{\|B\|}{\bar\mu}\right)\right]\varepsilon
\end{eqnarray*}
which was to be shown.
\end{proof}

\vspace{0.12in}

For $k=1,\,2,\,\ldots,\,n$ define the nonempty compact sets
$\C_{k}\subset\Complex^{q}$ as
\begin{eqnarray*}
\C_{k}:=\{w:\|w\|=1\,,\ (B_{kk}-\mu I_{q})w=0\ \mbox{for some}\
\mu\in\Real\,,\ \mbox{and}\
 \one_{q}^{T}w=0\}.
\end{eqnarray*}
Then define the nonnegative real number
\begin{eqnarray*}
\bar\gamma:=\left[\min_{k}\left(\min_{w\in\C_{k}}w^{*}G_{kk}w\right)\right]^{1/2}\,.
\end{eqnarray*}

\begin{lemma}\label{lem:gamma}
If ${\rm Re}\,\lambda_{2}(G_{kk}+jB_{kk})>0$ for all
$k=1,\,2,\,\ldots,\,n$ then $\bar\gamma>0$.
\end{lemma}

\begin{proof}
Suppose $\bar\gamma=0$. Then there should exist an index
$k\in\{1,\,2,\,\ldots,\,n\}$, a real number $\mu\in\Real$, and a
unit vector $w\in\Complex^{q}$ satisfying $B_{kk}w=\mu w$,
$\one_{q}^{T}w=0$, and $w^{*}G_{kk}w=0$. We have
$G_{kk},\,B_{kk}\in\setL(q,\,1)$ by Lemma~\ref{lem:Gkk}. Hence
$G_{kk}\one_{q}=0$ and $B_{kk}\one_{q}=0$. This allows us to write
\begin{eqnarray}\label{eqn:rolls}
(G_{kk}+jB_{kk})\one_{q}=0\,.
\end{eqnarray}
Since $G_{kk}$ is Laplacian we have $G_{kk}\geq 0$. Hence
$w^{*}G_{kk}w=0$ implies $G_{kk}w=0$ and we have
\begin{eqnarray}\label{eqn:royce}
(G_{kk}+jB_{kk})w=j\mu w\,.
\end{eqnarray}
Since $\one_{q}^{T}w=0$ the vectors $w$ and $\one_{q}$ must be
linearly independent. Then \eqref{eqn:rolls} and \eqref{eqn:royce}
imply that the matrix $G_{kk}+jB_{kk}$ must have at least two
eigenvalues on the imaginary axis. Also, due to $G_{kk},\,B_{kk}\geq
0$ all the eigenvalues of $G_{kk}+jB_{kk}$ must be on the closed
right half plane by Fact~\ref{fact:the}. This implies
$\lambda_{2}(G_{kk}+jB_{kk})=0$. Hence the result.
\end{proof}

\begin{theorem}\label{thm:mothership}
Suppose ${\rm Re}\,\lambda_{2}(G_{kk}+jB_{kk})>0$ for all
$k=1,\,2,\,\ldots,\,n$. Then there exists $r>0$ such that the
array~\eqref{eqn:arrayeps} synchronizes for all
$\varepsilon\in(0,\,r)$. In particular, one can choose
\begin{eqnarray*}
r=\frac{\bar\gamma\bar\sigma\bar\mu}{\left(\sqrt{\|G\|}+2\bar\gamma\right)
\sqrt{n-1}\|B\|(\bar\mu+\|B\|)}\,.
\end{eqnarray*}
\end{theorem}

\begin{proof}
We prove by contradiction. Let ${\rm
Re}\,\lambda_{2}(G_{kk}+jB_{kk})>0$ for all $k=1,\,2,\,\ldots,\,n$.
Then $\bar\gamma>0$ by Lemma~\ref{lem:gamma}. Let the coupling
strength
\begin{eqnarray}\label{eqn:something}
\varepsilon <
\frac{\bar\gamma}{\left(\sqrt{\|G\|}+2\bar\gamma\right) c}
\end{eqnarray}
be fixed, where we let
\begin{eqnarray*}
c= \frac{\sqrt{n-1}\|B\|}{\bar\sigma}\left(1+
\frac{\|B\|}{\bar\mu}\right)\,.
\end{eqnarray*}
Suppose however that the array~\eqref{eqn:arrayeps} fails to
synchronize. This implies, by Proposition~\ref{prop:swap}, ${\rm
Re}\,\lambda_{n+1}(\Omega_{\varepsilon})=0$ since all the
eigenvalues of $\Omega_{\varepsilon}$ are on the closed right half
plane by Fact~\ref{fact:the}. Let therefore
$\lambda_{n+1}(\Omega_{\varepsilon})=j\beta$ with $\beta\in\Real$.
For this eigenvalue we can find a unit vector
$\xi\in(\Complex^{q})^{n}$ satisfying
\begin{eqnarray}\label{eqn:eigen}
\Omega_{\varepsilon}\xi=j\beta\xi
\end{eqnarray}
and $\xi\notin{\rm span}\{[e_{1}\otimes \one_{q}],\,[e_{2}\otimes
\one_{q}],\,\ldots,\,[e_{n}\otimes \one_{q}]\}$; see the argument
employed in the proof of Lemma~\ref{lem:lambda2}. Without loss of
generality we assume the orthogonality
\begin{eqnarray}\label{eqn:orth}
[e_{k}\otimes \one_{q}]^{T}\xi=0\ \ \mbox{for all}\ \
k=1,\,2,\,\ldots,\,n\,.
\end{eqnarray}
Generality is not lost because using the symmetry
$\Omega_{\varepsilon}^{T}=\Omega_{\varepsilon}$ we can write
\begin{eqnarray*}\label{eqn:orth2}
j\beta [e_{k}\otimes \one_{q}]^{T}\xi =[e_{k}\otimes
\one_{q}]^{T}\Omega_{\varepsilon}\xi
=(\Omega_{\varepsilon}[e_{k}\otimes \one_{q}])^{T}\xi
=j\sigma_{k}[e_{k}\otimes \one_{q}]^{T}\xi
\end{eqnarray*}
which allows us to claim that if $j\beta\neq j\sigma_{k}$ for all
$k$ then \eqref{eqn:orth} must hold. If, on the other hand, $j\beta
= j\sigma_{\ell}$ for a particular $\ell$ then we can apply
Gram-Schmidt procedure to construct the new unit vector $\xi_{\rm
new}=(\xi-[e_{\ell}\otimes \one_{q}][e_{\ell}\otimes
\one_{q}]^{T}\xi)/\|\xi-[e_{\ell}\otimes \one_{q}][e_{\ell}\otimes
\one_{q}]^{T}\xi\|$, which indeed satisfies both \eqref{eqn:eigen}
and \eqref{eqn:orth}. By Lemma~\ref{lem:xi} there exist an index
$k\in\{1,\,2,\,\ldots,\,n\}$ and an eigenvector $w\in\Complex^{q}$
of $B_{kk}$ satisfying \eqref{eqn:ceps}. Without loss of generality
let this index be $k=1$. Also, let $\mu\in\Real$ be the
corresponding eigenvalue, i.e., $B_{11}w=\mu w$. Therefore we can
write $\xi=[e_{1}\otimes w]+\zeta$ for some
$\zeta\in(\Complex^{q})^{n}$ satisfying $\|\zeta\|\leq
c\varepsilon$. Whence
\begin{eqnarray}\label{eqn:w}
\|w\|
&=& \|e_{1}\otimes w\|\nonumber\\
&=& \|\xi-\zeta\|\nonumber\\
&\geq& \|\xi\|-\|\zeta\|\nonumber\\
&\geq& 1-c\varepsilon\,.
\end{eqnarray}
We now consider two cases.

{\em Case 1, $\mu\neq 0$}: By Lemma~\ref{lem:Gkk} we have
$B_{11}\in\setL(q,\,1)$. Hence $B_{11}\one_{q}=0$, i.e., $\one_{q}$
is an eigenvector whose eigenvalue is zero. This gives us
$\one_{q}^{T}w=0$ because a pair of eigenvectors of a real symmetric
matrix are orthogonal if the corresponding eigenvalues are
different. Note that \eqref{eqn:eigen} implies $G\xi=0$ (see the
proof of Lemma~\ref{lem:xi}). Using this, the lower bound
\eqref{eqn:w}, and the fact that $w/\|w\|\in\C_{1}$ we have
\begin{eqnarray*}
\bar\gamma^{2}(1-c\varepsilon)^{2}
&\leq& \bar\gamma^{2}\|w\|^{2}\\
&\leq& w^{*}G_{11}w\\
&=& [e_{1}\otimes w]^{*}G[e_{1}\otimes w]\\
&=& (\xi-\zeta)^{*}G(\xi-\zeta)\\
&=& \zeta^{*}G\zeta\\
&\leq& \|G\|\cdot\|\zeta\|^{2}\\
&\leq&\|G\|(c\varepsilon)^{2}
\end{eqnarray*}
which contradicts \eqref{eqn:something}.

{\em Case 2, $\mu=0$}: Since $[e_{1}\otimes \one_{q}]^{T}\xi=0$ by
\eqref{eqn:orth} we can write
\begin{eqnarray}\label{eqn:afterthought}
|\one_{q}^{T}w|
&=& |[e_{1}\otimes \one_{q}]^{T}[e_{1}\otimes w]|\nonumber\\
&=& |[e_{1}\otimes \one_{q}]^{T}(\xi-\zeta)|\nonumber\\
&=& |[e_{1}\otimes \one_{q}]^{T}\zeta|\nonumber\\
&\leq& \|e_{1}\otimes \one_{q}\|\cdot\|\zeta\|\nonumber\\
&=& \|\zeta\|\nonumber\\
&\leq& c\varepsilon\,.
\end{eqnarray}
Construct the vector $w_{1}\in\Complex^{q}$ as
\begin{eqnarray*}
w_{1}=w-\one_{q}\one_{q}^{T}w\,.
\end{eqnarray*}
Note that $B_{11}w_{1}=0$ (i.e., $w_{1}$ is an eigenvector of
$B_{11}$) and $\one_{q}^{T}w_{1}=0$. Also, by \eqref{eqn:w} and
\eqref{eqn:afterthought} we have
\begin{eqnarray*}
\|w_{1}\|&\geq& \|w\|-|\one_{q}^{T}w|\cdot\|\one_{q}\|\\
&\geq& 1-2c\varepsilon\,.
\end{eqnarray*}
This inequality, $G[e_{1}\otimes \one_{q}]=0$, $G\xi=0$, and
$w_{1}/\|w_{1}\|\in\C_{1}$ yield
\begin{eqnarray*}
\bar\gamma^{2}(1-2c\varepsilon)^{2}
&\leq& \bar\gamma^{2}\|w_{1}\|^{2}\\
&\leq& w_{1}^{*}G_{11}w_{1}\\
&=& [e_{1}\otimes w_{1}]^{*}G[e_{1}\otimes w_{1}]\\
&=& [e_{1}\otimes (w-\one_{q}\one_{q}^{T}w)]^{*}G[e_{1}\otimes (w-\one_{q}\one_{q}^{T}w)]\\
&=& (\xi-(\one_{q}^{T}w)[e_{1}\otimes \one_{q}]-\zeta)^{*}G(\xi-(\one_{q}^{T}w)[e_{1}\otimes \one_{q}]-\zeta)\\
&=& \zeta^{*}G\zeta\\
&\leq& \|G\|\cdot\|\zeta\|^{2}\\
&\leq& \|G\|(c\varepsilon)^{2}
\end{eqnarray*}
which contradicts \eqref{eqn:something}.
\end{proof}

\vspace{0.12in}

It is not difficult to see that $\lambda_{2}(G_{kk})>0$ implies
${\rm Re}\,\lambda_{2}(G_{kk}+jB_{kk})>0$. Hence:

\begin{corollary}
Suppose $\lambda_{2}(G_{kk})>0$ for all $k=1,\,2,\,\ldots,\,n$. Then
there exists $r>0$ such that the array~\eqref{eqn:arrayeps}
synchronizes for all $\varepsilon\in(0,\,r)$.
\end{corollary}

Consider now an array of coupled $n$-link pendulums where the
springs connect pairs of pendulums only through a particular link.
And likewise for the dampers, see Fig.~\ref{fig:pendcoupuni}. This
configuration makes a special case of \eqref{eqn:arrayeps} where the
coupling matrices are {\em commensurable}. That is, there exist
matrices $C_{\rm d}\in\Real^{m_{\rm d}\times n}$ and $C_{\rm
r}\in\Real^{m_{\rm r}\times n}$ such that for all $i,\,j$ we have
$D_{ij}=d_{ij}C_{\rm d}^{T}C_{\rm d}$ and $R_{ij}=r_{ij}C_{\rm
r}^{T}C_{\rm r}$ where $d_{ij},\,r_{ij}$ are nonnegative scalars.
This leads to the dynamics below, where the coupling enjoys a type
of uniformity,
\begin{eqnarray}\label{eqn:arrayepsuni}
M{\ddot x}_{i}+Kx_{i}+\sum_{j=1}^{q}d_{ij}C_{\rm d}^{T}C_{\rm
d}({\dot x}_{i}-{\dot x}_{j})+\varepsilon\sum_{j=1}^{q}r_{ij}C_{\rm
r}^{T}C_{\rm r}(x_{i}-x_{j})=0\,,\qquad i=1,\,2,\,\ldots,\,q\,.
\end{eqnarray}
Such uniformity makes the synchronization analysis significantly
simpler, yet not too simple to be interesting. Define the Laplacian
matrices $\ell_{\rm d},\,\ell_{\rm r}\in\setL(q,\,1)$ as
\begin{eqnarray*}
\ell_{\rm d}&:=&{\rm lap}\,(d_{ij})_{i,j=1}^{q}\,,\\
\ell_{\rm r}&:=&{\rm lap}\,(r_{ij})_{i,j=1}^{q}\,.
\end{eqnarray*}

\begin{figure}[h]
\begin{center}
\includegraphics[scale=0.55]{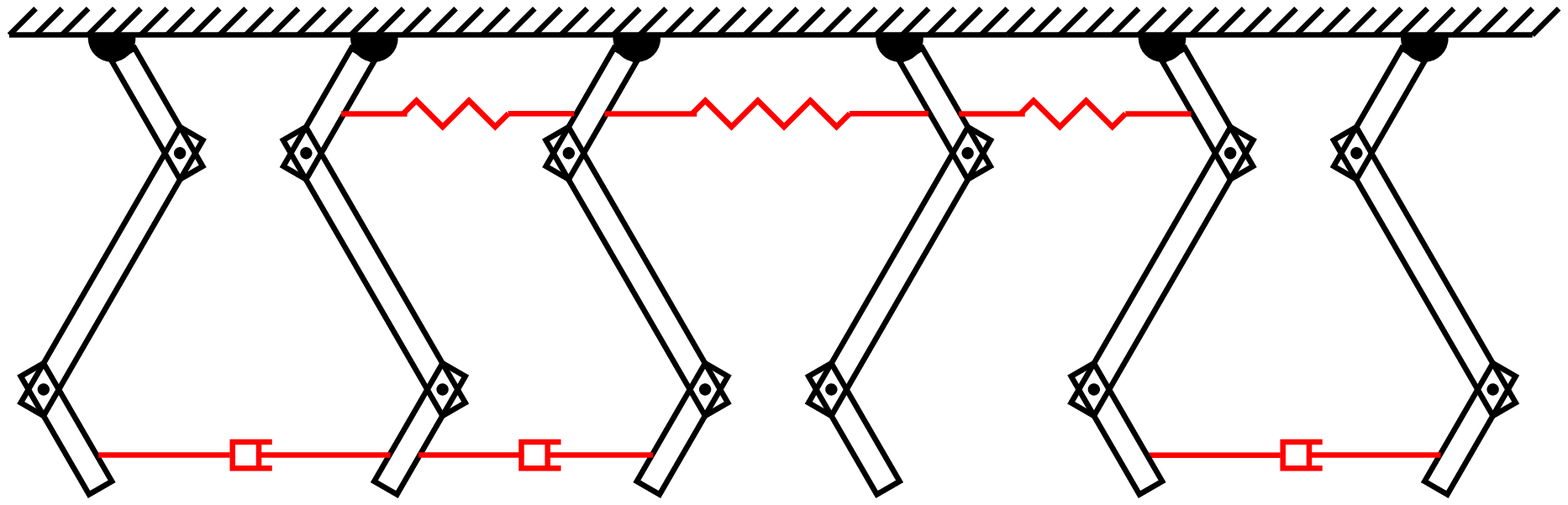}
\caption{Uniformly coupled 3-link pendulums.}\label{fig:pendcoupuni}
\end{center}
\end{figure}

\begin{corollary}\label{cor:ldpjlr}
Suppose ${\rm Re}\,\lambda_{2}(\ell_{\rm d}+j\ell_{\rm r})>0$ and
both $(C_{\rm d},\,M^{-1}K)$ and $(C_{\rm r},\,M^{-1}K)$ are
observable pairs. Then there exists $r>0$ such that the
array~\eqref{eqn:arrayepsuni} synchronizes for all
$\varepsilon\in(0,\,r)$.
\end{corollary}

\begin{proof}
We begin by proving the implication
\begin{eqnarray}\label{eqn:implication}
{\rm Re}\,\lambda_{2}(\ell_{\rm d}+j\ell_{\rm
r})>0\quad\implies\quad {\rm Re}\,\lambda_{2}(\alpha\ell_{\rm
d}+j\beta\ell_{\rm r})>0\ \ \mbox{for all scalars}\ \
\alpha,\,\beta>0\,.
\end{eqnarray}
Given $\alpha,\,\beta>0$, consider the matrix $\alpha\ell_{\rm
d}+j\beta\ell_{\rm r}$. Note that $\ell_{\rm d},\,\ell_{\rm r}\geq
0$ thanks to $\ell_{\rm d},\,\ell_{\rm r}\in\setL(q,\,1)$. Therefore
all the eigenvalues of $\alpha\ell_{\rm d}+j\beta\ell_{\rm r}$ are
on the closed right half plane by Fact~\ref{fact:the}. Also, since
$\ell_{\rm d}\one_{q}=0$ and $\ell_{\rm r}\one_{q}=0$, we have
$(\alpha\ell_{\rm d}+j\beta\ell_{\rm r})\one_{q}=0$. Hence, without
loss of generality, we can let $\lambda_{1}(\alpha\ell_{\rm
d}+j\beta\ell_{\rm r})=0$. Consider now the situation ${\rm
Re}\,\lambda_{2}(\alpha\ell_{\rm d}+j\beta\ell_{\rm r})\leq0$. This
implies $\lambda_{2}(\alpha\ell_{\rm d}+j\beta\ell_{\rm r})=j\gamma$
for some $\gamma\in\Real$. Let $\xi_{2}\in\Complex^{q}$ be the
corresponding unit eigenvector:
\begin{eqnarray}\label{eqn:yalincak}
(\alpha\ell_{\rm d}+j\beta\ell_{\rm r})\xi_{2}=j\gamma\xi_{2}\,.
\end{eqnarray}
If $j\gamma\neq 0$ then clearly we must have $\xi_{2}\notin{\rm
span}\,\{\one_{q}\}$. If $j\gamma=0$, on the other hand, then we can
choose $\xi_{2}\notin{\rm span}\,\{\one_{q}\}$. For if we could not
then $\one_{q}$ would have to be the only eigenvector for the
repeated eigenvalue at the origin. This would require that there
existed a generalized eigenvector $\zeta$ satisfying
$(\alpha\ell_{\rm d}+j\beta\ell_{\rm r})\zeta=\one_{q}$ which,
because $\alpha\ell_{\rm d}+j\beta\ell_{\rm r}$ is symmetric, would
lead to the following contradiction
\begin{eqnarray*}
1 =\one_{q}^{T}\one_{q}=\one_{q}^{T}(\alpha\ell_{\rm
d}+j\beta\ell_{\rm r})\zeta =((\alpha\ell_{\rm d}+j\beta\ell_{\rm
r})\one_{q})^{T}\zeta =0\,.
\end{eqnarray*}
Hence we let $\xi_{2}\notin{\rm span}\,\{\one_{q}\}$. Now,
left-multiplying \eqref{eqn:yalincak} by $\xi^{*}_{2}$ yields
$\alpha\xi^{*}_{2}\ell_{\rm d}\xi_{2}+j\beta\xi^{*}_{2}\ell_{\rm
r}\xi_{2}=j\gamma$ implying $\xi^{*}_{2}\ell_{\rm d}\xi_{2}=0$. This
in turn gives us $\ell_{\rm d}\xi_{2}=0$ because $\ell_{\rm d}\geq
0$. Therefore we have to have $\ell_{\rm
r}\xi_{2}=(\gamma/\beta)\xi_{2}$ by \eqref{eqn:yalincak}.
Consequently, $(\ell_{\rm d}+j\ell_{\rm
r})\xi_{2}=j(\gamma/\beta)\xi_{2}$. We also have $(\ell_{\rm
d}+j\ell_{\rm r})\one_{q}=0$. Since $\xi_{2}$ and $\one_{q}$ are
linearly independent, this means $\ell_{\rm d}+j\ell_{\rm r}$ has at
least two eigenvalues on the imaginary axis. Therefore we have
established ${\rm Re}\,\lambda_{2}(\alpha\ell_{\rm
d}+j\beta\ell_{\rm r})\leq0\implies {\rm Re}\,\lambda_{2}(\ell_{\rm
d}+j\ell_{\rm r})\leq0$, which gives us \eqref{eqn:implication}
because $\alpha,\,\beta$ were arbitrary.

Recall that $v_{1},\,v_{2},\,\ldots,\,v_{n}$ are the eigenvectors of
$P=M^{-1/2}KM^{-1/2}$, the corresponding eigenvalues being
$\sigma_{1},\,\sigma_{2},\,\ldots,\,\sigma_{n}$. Define now the
vectors $\tilde v_{k}=M^{-1/2}v_{k}$. These $\tilde v_{k}$ are the
eigenvectors of $M^{-1}K$ because we can write
\begin{eqnarray*}
M^{-1}K\tilde v_{k}
&=&M^{-1}KM^{-1/2}v_{k}\\
&=&M^{-1/2}(M^{-1/2}KM^{-1/2}v_{k})\\
&=&M^{-1/2}(\sigma_{k}v_{k})\\
&=&\sigma_{k}\tilde v_{k}\,.
\end{eqnarray*}
Define $\alpha_{k}=\|C_{\rm d}\tilde v_{k}\|^{2}$ and
$\beta_{k}=\|C_{\rm r}\tilde v_{k}\|^{2}$ for
$k=1,\,2,\,\ldots,\,n$. Recall $D_{ij}=d_{ij}C_{\rm d}^{T}C_{\rm d}$
and $R_{ij}=r_{ij}C_{\rm r}^{T}C_{\rm r}$. Starting from
\eqref{eqn:Gkklap} we can write
\begin{eqnarray*}
G_{kk}
&=&{\rm lap}\,(v_{k}^{T}M^{-1/2}D_{ij}M^{-1/2}v_{k})_{i,j=1}^{q}\\
&=&{\rm lap}\,(\tilde v_{k}^{T}C_{\rm d}^{T}C_{\rm d}\tilde v_{k}d_{ij})_{i,j=1}^{q}\\
&=&\|C_{\rm d}\tilde v_{k}\|^{2}{\rm lap}\,(d_{ij})_{i,j=1}^{q}\\
&=&\alpha_{k}\ell_{\rm d}\,.
\end{eqnarray*}
Likewise, we have $B_{kk}=\beta_{k}\ell_{\rm r}$. Hence, for all
$k=1,\,2,\,\ldots,\,n$,
\begin{eqnarray}\label{eqn:implication2}
G_{kk}+jB_{kk}=\alpha_{k}\ell_{\rm d}+j\beta_{k}\ell_{\rm r}\,.
\end{eqnarray}

Suppose now ${\rm Re}\,\lambda_{2}(\ell_{\rm d}+j\ell_{\rm r})>0$
and both $(C_{\rm d},\,M^{-1}K)$ and $(C_{\rm r},\,M^{-1}K)$ are
observable pairs. By PBH observability condition \cite{hespanha09}
we have to have $C_{\rm d}\tilde v_{k}\neq 0$ and $C_{\rm r}\tilde
v_{k}\neq 0$ for all $k=1,\,2,\,\ldots,\,n$. This means
$\alpha_{k},\,\beta_{k}>0$. The result then follows by
\eqref{eqn:implication}, \eqref{eqn:implication2}, and
Theorem~\ref{thm:mothership}.
\end{proof}

\section{Pure dissipative coupling}\label{sec:PD}

In the last part of our analysis we dispense with the restorative
coupling (e.g., springs connecting the pendulums) altogether and
focus on the special case of \eqref{eqn:array} where all $R_{ij}=0$.
This is the case where the coupling is purely dissipative:
\begin{eqnarray}\label{eqn:arraypuredis}
M{\ddot x}_{i}+Kx_{i}+\sum_{j=1}^{q}D_{ij}({\dot x}_{i}-{\dot
x}_{j})=0\,,\qquad i=1,\,2,\,\ldots,\,q\,.
\end{eqnarray}
The next result is closely related to \cite[Cor.~1]{tuna18}.

\begin{theorem}
The array~\eqref{eqn:arraypuredis} synchronizes if and only if
$\lambda_{2}(G_{kk})>0$ for all $k=1,\,2,\,\ldots,\,n$.
\end{theorem}

\begin{proof}
Define the matrix $\Omega_{0}=G+j[\Lambda\otimes I_{q}]$. Note that
the array~\eqref{eqn:arraypuredis} synchronizes if and only if ${\rm
Re}\,\lambda_{n+1}(\Omega_{0})>0$ thanks to
Remark~\ref{rem:observe}. Some of our earlier arguments on
$\Omega_{\varepsilon}$ are valid also on $\Omega_{0}$. By those
arguments we see that $\Omega_{0}[e_{k}\otimes
\one_{q}]=j\sigma_{k}[e_{k}\otimes \one_{q}]$ for
$k=1,\,2,\,\ldots,\,n$. That is, each $[e_{k}\otimes \one_{q}]$ is
an eigenvector, the corresponding eigenvalue being $j\sigma_{k}$.
Also, all the eigenvalues of $\Omega_{0}$ are on the closed right
half plane by Fact~\ref{fact:the}. Therefore we can let, without
loss of generality, $\lambda_{k}(\Omega_{0})=j\sigma_{k}$ for
$k=1,\,2\,\ldots,\,n$.

Suppose the array~\eqref{eqn:arraypuredis} fails to synchronize.
This implies $\lambda_{n+1}(\Omega_{0})=j\beta$ for some
$\beta\in\Real$. Let $\xi\in(\Complex^{q})^{n}$ be the corresponding
unit eigenvector. We can write
$j\beta=\xi^{*}G\xi+j\xi^{*}[\Lambda\otimes I_{q}]\xi$ by
\eqref{eqn:lambda}. This tells us (since $G\geq 0$ and
$[\Lambda\otimes I_{q}]>0$) that $\xi^{*}G\xi=0$ and, consequently,
$G\xi=0$. Therefore $[\Lambda\otimes I_{q}]\xi=\beta\xi$. That is,
$\xi$ is an eigenvector of $[\Lambda\otimes I_{q}]$ and $\beta$ an
eigenvalue. Now, $\Lambda={\rm
diag}\,(\sigma_{1},\,\sigma_{2},\,\ldots,\,\sigma_{n})$ implies
$\beta\in\{\sigma_{1},\,\sigma_{2},\,\ldots,\,\sigma_{n}\}$. Without
loss of generality let us take $\beta=\sigma_{1}$. Then $\xi$ has to
have the form $\xi=[e_{1}\otimes w]$ for some $w\in\Complex^{q}$.
Again without loss of generality we can further assume $w\notin{\rm
span}\,\{\one_{q}\}$. Generality is not lost; for, otherwise,
$[e_{1}\otimes \one_{q}]$ would be the only eigenvector of
$\Omega_{0}$ for the repeated eigenvalue $j\sigma_{1}$, which would
require the existence of a generalized eigenvector $\zeta$
satisfying $(\Omega_{0}-j\sigma_{1}I_{qn})\zeta=[e_{1}\otimes
\one_{q}]$. This however yields the contradiction below because
$\Omega_{0}^{T}=\Omega_{0}$
\begin{eqnarray*}
1=[e_{1}\otimes\one_{q}]^{T}[e_{1}\otimes\one_{q}]=[e_{1}\otimes\one_{q}]^{T}(\Omega_{0}-j\sigma_{1}I_{qn})\zeta
=((\Omega_{0}-j\sigma_{1}I_{qn})[e_{1}\otimes\one_{q}])^{T}\zeta=0\,.
\end{eqnarray*}
Note that we have
\begin{eqnarray*}
w^{*}G_{11}w=[e_{1}\otimes w]^{*}G[e_{1}\otimes w]=\xi^{*}G\xi=0
\end{eqnarray*}
implying $G_{11}w=0$ because $G_{11}\geq 0$ thanks to
$G_{11}\in\setL(q,\,1)$ by Lemma~\ref{lem:Gkk}. Furthermore,
$G_{11}\one_{q}=0$. Since the set $\{w,\,\one_{q}\}$ is linearly
independent the eigenvalue of $G_{11}$ at the origin must be
repeated. This means $\lambda_{2}(G_{11})=0$ because $G_{11}\geq 0$.

To show the other direction let us suppose this time that
$\lambda_{2}(G_{\ell\ell})\leq 0$ for some $\ell$. Being a Laplacian
matrix, $G_{\ell\ell}\geq 0$ and $G_{\ell\ell}\one_{q}=0$. Therefore
the eigenvalue at the origin is repeated and there exists a vector
$u\notin{\rm span}\,\{\one_{q}\}$ satisfying $G_{\ell\ell}u=0$.
Construct now the vector $\eta=[e_{\ell}\otimes u]$. Clearly, this
vector satisfies
\begin{eqnarray}\label{eqn:linind}
\eta\notin{\rm span}\,\{[e_{1}\otimes \one_{q}],\,[e_{2}\otimes
\one_{q}],\,\ldots,\,[e_{n}\otimes \one_{q}]\}\,.
\end{eqnarray}
Moreover,
\begin{eqnarray*}
\eta^{*}G\eta=[e_{\ell}\otimes u]^{*}G[e_{\ell}\otimes
u]=u^{*}G_{\ell\ell}u=0
\end{eqnarray*}
which, since $G\geq 0$, implies $G\eta=0$. This allows us to see
that $\eta$ is an eigenvector of $\Omega_{0}$ because
\begin{eqnarray}\label{eqn:linind2}
\Omega_{0}\eta
&=&(G+j[\Lambda\otimes I_{q}])\eta\nonumber\\
&=&j[\Lambda\otimes I_{q}][e_{\ell}\otimes u]\nonumber\\
&=&j[(\Lambda e_{\ell})\otimes (I_{q}u)]\nonumber\\
&=&j[(\sigma_{\ell} e_{\ell})\otimes (I_{q}u)]\nonumber\\
&=&j\sigma_{\ell}[e_{\ell}\otimes u]\nonumber\\
&=&j\sigma_{\ell}\eta\,.
\end{eqnarray}
Now, \eqref{eqn:linind} and \eqref{eqn:linind2} tell us that
$\Omega_{0}$ has at least $n+1$ linearly independent eigenvectors
whose eigenvalues lie on the imaginary axis. But this implies ${\rm
Re}\,\lambda_{n+1}(\Omega_{0})=0$. Hence the result.
\end{proof}

\vspace{0.12in}

Consider now the scenario where the coupling in the
array~\eqref{eqn:arraypuredis} is uniform. That is, there exists a
matrix $C_{\rm d}\in\Real^{m_{\rm d}\times n}$ such that
$D_{ij}=d_{ij}C_{\rm d}^{T}C_{\rm d}$ where $d_{ij}$ are nonnegative
scalars. Under this condition the array dynamics take the form
\begin{eqnarray}\label{eqn:arraypuredisuni}
M{\ddot x}_{i}+Kx_{i}+\sum_{j=1}^{q}d_{ij}C_{\rm d}^{T}C_{\rm
d}({\dot x}_{i}-{\dot x}_{j})=0\,,\qquad i=1,\,2,\,\ldots,\,q\,.
\end{eqnarray}
The coupling of this array is represented by two parameters: the
Laplacian matrix $\ell_{\rm d}={\rm lap}\,(d_{ij})_{i,j=1}^{q}$ and
the output matrix $C_{\rm d}$. How they are linked to
synchronization is stated next.

\begin{corollary}
The array~\eqref{eqn:arraypuredisuni} synchronizes if and only if
$\lambda_{2}(\ell_{\rm d})>0$ and $(C_{\rm d},\,M^{-1}K)$ is
observable.
\end{corollary}

\begin{proof}
The demonstration is similar to that of Corollary~\ref{cor:ldpjlr}.
\end{proof}

\section{Conclusion}

In this paper we studied the problem of synchronization in an array
of identical oscillators subject to both dissipative and restorative
coupling. We presented a simple way to combine the pair of
matrix-weighted Laplacians (one representing the dissipative, the
other the restorative coupling) in a single complex-valued matrix
and established an equivalence relation between a certain spectral
property of this matrix and the collective behavior of the
oscillators. Also, we projected this method to generate more refined
conditions for synchronization applicable when the restorative
coupling is either weak or absent altogether.

\bibliographystyle{plain}
\bibliography{references}

\begin{thebibliography}{10}

\bibitem{arnold89}
V.I. Arnold.
\newblock {\em Mathematical Methods of Classical Mechanics (Second Edition)}.
\newblock Springer, 1989.

\bibitem{desoer69}
C.A. Desoer and E.S. Kuh.
\newblock {\em Basic Circuit Theory}.
\newblock McGraw-Hill, 1969.

\bibitem{dorfler14}
F.~Dorfler and F.~Bullo.
\newblock Synchronization in complex networks of phase oscillators: {A} survey.
\newblock {\em Automatica}, 50:1539--1564, 2014.

\bibitem{eroglu17}
D.~Eroglu, J.S.W. Lamb, and T.~Pereira.
\newblock Synchronisation of chaos and its applications.
\newblock {\em Contemporary Physics}, 58:207--243, 2017.

\bibitem{golub96}
G.H. Golub and C.F.~Van Loan.
\newblock {\em Matrix Computations (Third Edition)}.
\newblock The Johns Hopkins University Press, 1996.

\bibitem{hespanha09}
J.P. Hespanha.
\newblock {\em Linear Systems Theory}.
\newblock Princeton, 2009.

\bibitem{khalil96}
H.K. Khalil.
\newblock {\em Nonlinear Systems (Second Edition)}.
\newblock Prentice Hall, 1996.

\bibitem{lax96}
P.D. Lax.
\newblock {\em Linear Algebra}.
\newblock John Wiley \& Sons, 1996.

\bibitem{li10}
Z.~Li, Z.~Duan, G.~Chen, and L.~Huang.
\newblock Consensus of multi-agent systems and synchronization of complex
  networks: {A} unified viewpoint.
\newblock {\em IEEE Transactions on Circuits and Systems I: Regular Papers},
  57:213--224, 2010.

\bibitem{olfati04}
R.~Olfati-Saber and R.M. Murray.
\newblock Consensus problems in networks of agents with switching topology and
  time delays.
\newblock {\em IEEE Transactions on Automatic Control}, 49:1520--1533, 2004.

\bibitem{ren08}
W.~Ren.
\newblock Synchronization of coupled harmonic oscillators with local
  interaction.
\newblock {\em Automatica}, 44:3195--3200, 2008.

\bibitem{su09}
H.~Su, X.~Wang, and Z.~Lin.
\newblock Synchronization of coupled harmonic oscillators in a dynamic
  proximity network.
\newblock {\em Automatica}, 45:2286--2291, 2009.

\bibitem{sun15}
W.~Sun, X.~Yu, J.~Lu, and S.~Chen.
\newblock Synchronization of coupled harmonic oscillators with random noises.
\newblock {\em Nonlinear Dynamics}, 79:473--484, 2015.

\bibitem{trinh18}
M.H. Trinh, C.V. Nguyen, Y.-H. Lim, and H.-S. Ahn.
\newblock Matrix-weighted consensus and its applications.
\newblock {\em Automatica}, 89:415--419, 2018.

\bibitem{tuna16}
S.E. Tuna.
\newblock Synchronization under matrix-weighted {L}aplacian.
\newblock {\em Automatica}, 73:76--81, 2016.

\bibitem{tuna17}
S.E. Tuna.
\newblock Synchronization of harmonic oscillators under restorative coupling
  with applications in electrical networks.
\newblock {\em Automatica}, 75:236--243, 2017.

\bibitem{tuna18}
S.E. Tuna.
\newblock Observability through a matrix-weighted graph.
\newblock {\em IEEE Transactions on Automatic Control}, 63:2061--2074, 2018.

\bibitem{zhou12}
J.~Zhou, H.~Zhang, L.~Xiang, and Q.~Wu.
\newblock Synchronization of coupled harmonic oscillators with local
  instantaneous interaction.
\newblock {\em Automatica}, 48:1715--1721, 2012.

\end{thebibliography}
\end{document}